\documentclass[a4paper,12pt]{article}

\usepackage[intlimits]{amsmath}
\usepackage{amssymb}
\usepackage{amsthm}
\usepackage{mathrsfs}
\usepackage{cite}
\usepackage{color}
\usepackage{bm}

\newcommand{\pr}{\partial}
\newcommand{\dom}{\Omega}

\newcommand{\R}{\mathbb{R}}

\newcommand{\N}{\mathbb{N}}

\newcommand{\dd}{\,\text{d}}

\newtheorem{thm}{Theorem}
\newtheorem{prop}{Proposition}
\newtheorem{cor}{Corollary}

\title{Two uniqueness results in the inverse boundary value problem for the weighted p-Laplace equation}
\author{C\u{a}t\u{a}lin I. C\^{a}rstea\thanks{Department of Applied Mathematics, National Yang Ming Chiao Tung University, Hsinchu 300, Taiwan, R.O.C.; \mbox{email: catalin.carstea@gmail.com}} \and Ali Feizmohammadi\thanks{Department of Mathematics, University of Toronto, Mississauga, ON L5L 1C6, Canada; email: ali.feizmohammadi@utoronto.ca}}
\date{}

\begin{document}
\maketitle

\begin{abstract}
In this paper we prove a general uniqueness result in the inverse boundary value problem for the weighted p-Laplace equation in the plane, with smooth weights. We also prove a uniqueness result in dimension 3 and higher, for real analytic weights that are subject to a smallness condition on one of their directional derivatives. Both results are obtained by linearizing the equation at a solution without critical points. This unknown solution is then recovered, together with the unknown weight. 
\end{abstract}

\section{Introduction}

Let $\dom\subset\R^n$, $n\geq2$, be a compact connected set with nonempty interior and a smooth boundary, let $\gamma\in C^\infty(\bar \dom)$ be a positive function, and finally let $p\in(1,2)\cup(2,\infty)$. We consider the boundary value problem
\begin{equation}\label{eq}
\left\{\begin{array}{l}\nabla\cdot(\gamma|\nabla u|^{p-2}\nabla u)=0,\\[7pt] u|_{\pr\dom}=f,\end{array}\right.
\end{equation}
where $f$, $u$ are real valued functions. Equation \eqref{eq} is known as the \emph{weighted $p$-Laplace} equation and it is a quasilinear, degenerate elliptic equation. The forward problem for this equation is well studied and we have
\begin{thm}[e.g. \mbox{\cite[Theorem 1]{Lie}}]\label{thm-existence}
Let $f\in C^{1,\alpha}(\pr\dom)$ for some $\alpha \in (0,1]$. There exist $\beta\in (0,1)$ and $C(\|f\|_{C^{1,\alpha}(\pr\dom)})>0$ nondecreasing such that equation \eqref{eq} has a unique weak solution $u\in C^{1,\beta}(\bar\dom)$ and
\begin{equation}
\|u\|_{C^{1,\beta}(\bar\dom)}\leq C(\|f\|_{C^{1,\alpha}(\pr\dom)}).
\end{equation}
\end{thm}

\noindent It is therefore possible to define the Dirichlet-to-Neumann map associated to \eqref{eq} by
\begin{equation}
\Lambda_\gamma(f)=\left.\left(\gamma|\nabla u|^{p-2}\partial_\nu u\right)\right|_{\pr\dom}, \qquad \forall\, f \in C^{1,\alpha}(\pr\dom),
\end{equation}
where $u$ is the unique solution to \eqref{eq} and $\nu$ is the exterior normal unit vector on $\pr \dom$. 

In \cite{Ca}, Calder\'on proposed the following question/inverse problem: can an elliptic coefficient $\gamma$ be recovered from the Dirichlet-to-Neumann map associated to the equation $\nabla\cdot(\gamma \nabla u)=0$? A positive answer for general smooth coefficients $\gamma$ was first provided in \cite{SyUh} in dimension 3 or higher, and by \cite{Na} in the plane.  In the intervening decades, similar questions for other equations have been investigated  in a large number of papers. It is beyond the purposes of our work to give a full account of the existing inverse problems literature. We will reference below those works that are most closely related to our own, in terms of subject matter or technique.

In this paper we are interested in the inverse problem of recovering the a priori unknown coefficient $\gamma$ in \eqref{eq}, given the knowledge of $\Lambda_\gamma$. This a natural analogue of the original problem of Calder\'on. 
We will prove two  results. The first is the following general uniqueness result in the plane.
\begin{thm}\label{thm-main-1}
Let $n=2$ and let $\gamma, \tilde\gamma\in C^\infty(\bar\dom)$ be strictly positive functions. If
 $\Lambda_{\gamma}=\Lambda_{\tilde \gamma}$, then $\gamma=\tilde\gamma$.
\end{thm}
In dimensions 3 and higher  we prove the following uniqueness result for real-analytic weights..
\begin{thm}\label{thm-main-2}
Let $n\geq 3$ and let $\zeta \in \R^n$ be a unit vector. Suppose there exists a point $z\in\pr\dom$ in a neighborhood of which $\pr\dom$ is flat. There exists $\mu>0$, depending only on $\Omega$ and $n$, such that if $\gamma, \tilde\gamma\in C^\omega(\bar\dom)$ are strictly positive functions with $\|\zeta\cdot\nabla\gamma\|_{C^{0,\alpha}(\dom)}$, $\|\zeta\cdot\nabla\tilde\gamma\|_{C^{0,\alpha}(\dom)}<\mu$, then $\Lambda_{\gamma}=\Lambda_{\tilde\gamma}$ implies $\gamma=\tilde\gamma$.
\end{thm}

The study of inverse problems for non-linear equations is not new, but in recent years there has been a considerable increase in the interest for this topic. As examples, we can cite the papers   \cite{FeOk}, \cite{Is1}, \cite{IsNa}, \cite{IsSy},  \cite{KrUh}, \cite{KrUh2}, \cite{LaLiiLinSa1}, \cite{LaLiiLinSa2}, \cite{Sun2} on semilinear equations, and \cite{CarFe1}, \cite{CarFe2}, \cite{CarFeKiKrUh}, \cite{CarKa}, \cite{CarNaVa}, \cite{Car}, \cite{EgPiSc}, \cite{HeSun}, \cite{Is2}, \cite{KaNa}, \cite{MuUh}, \cite{Sh}, \cite{Sun1}, \cite{Sun3}, \cite{SunUh} on quasilinear equations. By and large, all these works rely on a so called \emph{second/higher linearization} method, which first appeared in \cite{Is1} . This method consists of of using Dirichlet data that depends on a small (or large) parameter $\epsilon$, typically of the form $\epsilon\phi$ (or $\lambda+\epsilon\phi$, with $\lambda$ a constant, if constants are a solution  to the linear part of the equation). One then uses the asymptotic expansion of the Dirichlet-to-Neumann map in terms of the parameter $\epsilon$ to obtain information about the coefficients of the equation. Sometimes this is presented as differentiating the equation with respect to the small parameter, then setting it to zero.

Our paper is not the first to take up the inverse boundary value problem for the weighted p-Laplacian. The works  \cite{BrKaSa}, \cite{Br}, \cite{BrHaKaSa},  \cite{BrIlKa}, \cite{GuKaSa}, \cite{KaWa}, \cite{SaZh} all address aspects of the same problem. We note that a uniqueness result without additional constraints, such as monotonicity, has not yet been previously derived for the weighted p-Laplacian. Also, past boundary determination results have only yielded $\gamma|_{\pr\dom}$ and $\pr_\nu\gamma|_{\pr\dom}$, but not the rest of the derivatives of $\gamma$ on the boundary (see \cite{SaZh}, \cite{Br}). For other degenerate equations, the only known results are those of \cite{CarGhNa}, \cite{CarGhUh} where general uniqueness results are derived for the coefficients of porous medium equations.

The approach to the proofs of Theorems \ref{thm-main-1} and \ref{thm-main-2} also makes use of a linearization method. In equation \eqref{eq} we use Dirichlet data of the form $f=\phi_0+\epsilon\phi$, with $\epsilon$ a small parameter. Let $u_\epsilon$ be the corresponding solution and, assuming we are justified in taking the derivative, let $\dot u=\left.\frac{\dd}{\dd\epsilon} u_\epsilon\right|_{\epsilon=0}$. Further assuming we can differentiate the equation, it is not hard to see that $\dot u$ should satisfy the  anisotropic, linear equation
\begin{equation}\label{eq-A}
\left\{\begin{array}{l}\nabla\cdot(A\nabla \dot u)=0,\\[7pt] \dot u|_{\pr\dom}=\phi,\end{array}\right.
\end{equation}
where $A$ is the matrix with the $u_0$-dependent coefficients
\begin{equation}
A_{jk}=\gamma|\nabla u_0|^{p-2}\left(\delta_{jk}+(p-2)\frac{\pr_ju_0\pr_ku_0}{|\nabla u_0|^2}\right).
\end{equation}
The Dirichlet-to-Neumann map $\Lambda_\gamma$  determines the Dirichlet-to-Neumann map $\Lambda_A$ for the equation \eqref{eq-A}. 

In order to use already established results for the determination of the cofficient matrix $A$, we need it to be elliptic. Indeed, even the differentiability of $u_\epsilon$ w.r.t. $\epsilon$ is in question unless that is the case. We then see that the unknown $u_0$ must be guaranteed to have no critical points. In dimension 2, by results of Alessandrini and Sigalotti in \cite{AlSi}, we can guarantee the absence of critical points by choosing Dirichlet data that has single local minimum and maximum points on $\pr\dom$. In dimension 3 and higher something like this is unlikely to hold, as even for linear elliptic equations it is known that for each Dirichlet data there is an open set of smooth coefficients that produce solutions with critical points  (see \cite{AlBaDi}). We can show however that, for coefficients $\gamma$ that vary slowly in one direction, there exists explicit Dirichlet data for which no critical points appear.

We can also point out here a simple corollary of our linearization result (Proposition \ref{prop-linearization} below), for weights that are constant in one direction. 
\begin{cor}
Let $n\geq 3$ and $\zeta\in\R^n$ be a unit vector. If $\gamma, \tilde\gamma\in C^\infty(\dom)$ are such that $\zeta\cdot\gamma=\zeta\cdot\tilde\gamma=0$ and $\Lambda_{\gamma}=\Lambda_{\tilde\gamma}$, then $\gamma=\tilde\gamma$.
\end{cor}
\begin{proof}
In this case $u_0=\zeta\cdot x$ is a solution of \eqref{eq}, with either weight. Then
\begin{equation}
A_{jk}=\gamma\left(\delta_{jk}+(p-2)\zeta_j\zeta_k\right),\quad \tilde A_{jk}=\tilde \gamma\left(\delta_{jk}+(p-2)\zeta_j\zeta_k\right).
\end{equation}
After a rescaling in the $\zeta$ direction, the linearized problem reduces to the classical Calder\'on problem with isotropic conductivities.
\end{proof}

The linearization procedure is detailed in in section \ref{sec_lin}. In section \ref{sec_plane} we give a proof of Theorem \ref{thm-main-1}. By the well known result \cite{Na} of Nachman, we have uniqueness for the coefficient matrix $A$, up to diffeomorphism invariance. Making use of the particular structure of $A$, we then succeed in showing that the diffeomorphism relating $A$ and $\tilde A$ must be trivial and that $\gamma=\tilde\gamma$. In section \ref{sec_higher} we give the proof of Theorem \ref{thm-main-2}. Our approach is to use boundary determination results for equation \eqref{eq-A} to obtain the values of all tangential directions of $A$ on the boundary, together with all their normal direction derivatives. From this information we are then able to inductively show uniqueness for the values of all the normal direction derivatives $\pr_\nu^k u_0|_{\pr\dom}$, $\pr_\nu^k\gamma|_{\pr\dom}$, $k=0,1,2,\ldots$ . Since here $\gamma$ is assumed to be a real-analytic function, this is enough to recover it on $\dom$.

\section{Linearizing the $p$-Laplace equation}
\label{sec_lin}

For each $\xi\in\R^n\setminus\{0\}$ let
\begin{equation}
J_j(\xi)=|\xi|^{p-2}\xi_j,\quad j=1,\ldots,n.
\end{equation}
Then
\begin{equation}
\frac{\partial}{\partial\xi_k}J_j(\xi)=|\xi|^{p-2}\left(\delta_{jk}+(p-2)\frac{\xi_j\xi_k}{|\xi|^2}\right).
\end{equation}
In what follows, we will repeatedly use the Taylor's formula
\begin{equation}\label{Taylor}
J_j(\zeta)=J_j(\xi)+\sum_{k=1}^n(\zeta_k-\xi_k)\int_0^1\partial_{\xi_k}J(\xi+t(\zeta-\xi))\dd t.
\end{equation}
We plan to linearize equation \eqref{eq} near some solution $u_0$, whose boundary data $u_0|_{\pr\dom}=\phi_0$ is known. As will become apparent below, we can only perform the linearization if $u_0$ does not have any critical points in $\dom$. In dimension two plenty of such solutions exist, thanks to the following proposition due to Alessandrini and Sigalotti.

\begin{prop}[see \mbox{\cite[Theorem 5.1]{AlSi}}]\label{prop-u0-2}
If $n=2$ there exists boundary data $\phi_0\in C^\infty(\partial \dom)$ independent of $\gamma$ such that the corresponding solution $u_0$ of \eqref{eq} is in $C^\infty(\bar\dom)$ and  $|\nabla u_0(x)|>0$ for any $x\in\dom$.
\end{prop}
In higher dimensions, even for a linear elliptic equation with unknown coefficients it is impossible to guarantee the absence of critical points (see \cite{AlBaDi}). We can still show the existence of such a solution provided the weight $\gamma$ is sufficiently close to a constant.
\begin{prop}\label{prop-u0-3+}
Let $\zeta \in \R^n$ be a unit vector. There exists $\mu>0$ so that if $\|\zeta\cdot\nabla\gamma\|_{C^{0,\alpha}(\dom)}<\mu$, then there exists $u_0\in C^\infty(\bar\dom)$ which solves \eqref{eq} with  boundary data $\phi_0=\zeta\cdot x$, and  is such that $|\nabla u_0(x)|>0$ for any $x\in\dom$. 
\end{prop}
\begin{proof}
Without loss of generality, we assume that $\zeta=(1,0,\ldots,0)$. We make the ansatz
\begin{equation}
u_0(x)= x_1+R,\quad R|_{\pr\dom}=0.
\end{equation}
By \eqref{Taylor} we have
\begin{equation}
\sum_k B_{jk}(\nabla R)\partial_k R=\gamma J_j(\nabla u_0)-\gamma\delta_{1j},
\end{equation}
\begin{equation}
B_{jk}(\xi)=\gamma\int_0^1| e_1+t\xi|^{p-2}\left(\delta_{jk}+(p-2)\frac{(\delta_{1j}+t\xi_j)(\delta_{1k}+t\xi_k)}{|e_1+t\xi|^2}\right)\dd t.
\end{equation}
Taking the divergence of the above we get
\begin{equation}
\left\{\begin{array}{l}\nabla\cdot(B(\nabla R)\nabla  R)=-\partial_1\gamma,\\[7pt]  R|_{\pr\dom}=0.\end{array}\right.
\end{equation}
Let $V\in C^{2,\alpha}(\bar\dom)$ be such that $\|V\|_{C^{2,\alpha}(\dom)}<1/2$ and define the map $T(V)=U$, where $U$ is the solution to
\begin{equation}
\left\{\begin{array}{l}\nabla\cdot(B(\nabla V)\nabla  U)=-\partial_1\gamma,\\[7pt]  U|_{\pr\dom}=0.\end{array}\right.
\end{equation}
Since $B(\nabla V)\in C^{1,\alpha}$ are uniformly elliptic coefficients, it follows that a unique solution $U\in C^{2,\alpha}(\bar\dom)$ exists (see \cite[Theorem 6.14]{GiTru}). Furthermore, by \cite[Theorem 6.6]{GiTru} we have
\begin{equation}
\|U\|_{C^{2,\alpha}(\dom)}\leq C \|\pr_1\gamma\|_{C^{0,\alpha}(\dom)}.
\end{equation}
If the right hand side is less than $1/2$,
by Shauder's fixed point theorem (see \cite[Theorem 11.1]{GiTru}) it follows that $T$ has a fixed point on the ball of radius $1/2$ in $C^{2,\alpha}(\bar\dom)$. By uniqueness of solutions for \eqref{eq}, this must be $R$ and we conclude that 
\begin{equation}
\|\nabla R\|_{L^\infty(\dom)}\leq \frac{1}{2},
\end{equation}
so 
\begin{equation}
|\nabla u_0(x)|>\frac{1}{2},\quad\forall x\in\dom.
\end{equation}

Note that the nonvanishing of the gradient $\nabla u_0$ makes the equation satisfied by $u_0$ elliptic, so the smoothness of $u_0$ follows.
\end{proof}

In what follows we will assume that $u_0$ is as in the preceding two propositions. Let $A$ be the matrix with coefficients
\begin{equation}
A_{jk}=\gamma|\nabla u_0|^{p-2}\left(\delta_{jk}+(p-2)\frac{\pr_ju_0\pr_ku_0}{|\nabla u_0|^2}\right).
\end{equation}

\begin{prop}\label{prop-linearization}
Under the assumptions of either Proposition \ref{prop-u0-2} or Proposition \ref{prop-u0-3+}, we have that the Dirichlet-to-Neumann map $\Lambda_\gamma$ for the weighted p-Laplace equation \eqref{eq} determines the Dirichlet-to-Neumann map $\Lambda_A$ for the linear equation $\nabla\cdot(A\nabla u)=0$, on the same domain $\dom$. 
\end{prop}
\begin{proof}
 For $\phi\in C^\infty(\pr\dom)$, and $\epsilon\in\R$ small, let $u_\epsilon$ be the solution of
\begin{equation}\label{eq-epsilon}
\left\{\begin{array}{l}\nabla\cdot(\gamma|\nabla u_\epsilon|^{p-2}\nabla u_\epsilon)=0,\\[7pt] u_\epsilon|_{\pr\dom}=u_0|_{\pr\dom}+\epsilon\phi.\end{array}\right.
\end{equation}
We make the  Ansatz
\begin{equation}
u_\epsilon=u_0+R_\epsilon.
\end{equation}
By Theorem \ref{thm-existence}
 and the theorem of Arzel\`a-Ascoli,  it follows that (on a subsequence) we have that $R_\epsilon\to R_0$ in $C^1(\dom)$. Since then $u_0+R_0$ would be a weak solution of the same boundary value problem $u_0$ satisfies, it follows that $R_\epsilon\to0$ in $C^1(\dom)$. Since the limit is the same for every subsequence, it follows that in fact we do not need to pass to a subsequence.

By Taylor's formula we have that
\begin{equation}\label{Taylor-1}
\sum_k\partial_kR_\epsilon\int_0^1\partial_{\xi_k}J_j(\nabla u_0+t\nabla R_\epsilon)\dd t=J_j(\nabla u_\epsilon)-J_j(\nabla u_0).
\end{equation}
Let 
\begin{equation}
A^\epsilon_{jk}=\gamma\int_0^1\partial_{\xi_k}J_j(\nabla u_0+t\nabla R_\epsilon)\dd t.
\end{equation}
Since $R_\epsilon\to0$ in $C^1$, it follows that $|\nabla u_0+t\nabla R_\epsilon|$ is uniformly bounded and uniformly bounded away from zero. This implies that $A^\epsilon_{jk}$ is a set of elliptic coefficients, with ellipticity bounds independent of $\epsilon$. Taking gradients in \eqref{Taylor-1} we get that $R_\epsilon$ satisfies
\begin{equation}\label{eq-r-e}
\left\{\begin{array}{l}\nabla\cdot(A^\epsilon\nabla  R_\epsilon)=0,\\[7pt]  R_\epsilon|_{\pr\dom}=\epsilon\phi,\end{array}\right.
\end{equation}
It follows that 
\begin{equation}
\|R_\epsilon\|_{C^{1,\beta}(\dom)}\leq C\epsilon.
\end{equation}
We can again invoke the theorem of  Arzela-Ascoli to conclude that there must exist $\dot u\in C^1(\dom)$ such that $\epsilon^{-1}R_\epsilon\to\dot u$ in $C^1(\dom)$. Taking the limit in \eqref{eq-r-e} we see that $\dot u$ must be a weak solution of
\begin{equation}\label{eq-dot-u}
\left\{\begin{array}{l}\nabla\cdot(A\nabla \dot u)=0,\\[7pt] \dot u|_{\pr\dom}=\phi,\end{array}\right.
\end{equation}

Returning to \eqref{Taylor-1}, dividing by $\epsilon$ and taking the limit $\epsilon\to0$, we have that
\begin{multline}
\nu\cdot A\nabla\dot u=\lim_{\epsilon\to0}\frac{\nu\cdot J(\nabla u_\epsilon)-\nu\cdot J(\nabla u_0)}{\epsilon}\\[7pt]
=\lim_{\epsilon\to0}\frac{\Lambda_{\gamma}(u_0|_{\pr\dom}+\epsilon\phi)-\Lambda_{\gamma}(u_0|_{\pr\dom})}{\epsilon}.
\end{multline}
We see then that the Neumann data for the equation \eqref{eq-dot-u} is determined by the map $\Lambda_\gamma$.
\end{proof}

\section{Proof of Theorem \ref{thm-main-1}}\label{sec_plane}

Suppose $n=2$ and we have $\gamma$, $\tilde\gamma$ as above such that $\Lambda_\gamma=\Lambda_{\tilde\gamma}$.  We use notation such as $u_\epsilon$, $\tilde u_\epsilon$ to denote the corresponding solutions to \eqref{eq-epsilon}, etc. Observe that one consequence of the identity of the DN maps is
\begin{equation}
\int_\dom \gamma|\nabla u_0|^p\dd x=\int_\dom \tilde\gamma|\nabla \tilde u_0|^p\dd x.
\end{equation}

By Proposition \ref{prop-linearization} we have that $\Lambda_A=\Lambda_{\tilde A}$. From \cite[Theorem 2]{Na} it follows that there must exist a diffeomorphism $\Phi:\bar\dom\to\bar\dom$ such that $\Phi|_{\pr\dom}=Id$ and
\begin{equation}\label{linear-uniqueness}
\tilde A(x)=\frac{1}{|D\Phi|} (D\Phi)^TA\, D\Phi\circ\Phi^{-1}(x).
\end{equation}

Note that
\begin{multline}
\det\left(\delta_{jk}+(p-2)\frac{\pr_j u_0\pr_k u_0}{|\nabla u_0|^2}\right)\\[7pt]
=\left(1+(p-2)\frac{(\pr_1 u_0)^2}{|\nabla u_0|^2}\right)\left(1+(p-2)\frac{(\pr_2 u_0)^2}{|\nabla u_0|^2}\right)\\[7pt]
-(p-2)^2\frac{(\pr_1 u_0\pr_2 u_0)^2}{|\nabla u_0|^4}
=p-1.
\end{multline}
Taking determinants on both sides of \eqref{linear-uniqueness} we obtain
\begin{equation}
\left(\tilde\gamma|\nabla \tilde u_0|^{p-2}\right)^2=\left(\gamma |\nabla u_0|^{p-2}\right)^2|D\Phi|^{-2}|D\Phi|^2\circ\Phi^{-1},
\end{equation}
so 
\begin{equation}\label{eq-p-2}
\tilde\gamma|\nabla\tilde u_0|^{p-2}=\gamma|\nabla u_0|^{p-2}\circ\Phi^{-1}.
\end{equation}

Another consequence of $\Lambda_A=\Lambda_{\tilde A}$ is that for each $\phi$ we have  $\dot{\tilde u}=\dot{u}\circ\Phi^{-1}$. Incidentally, for $\phi=u_0|_{\pr\dom}$ the solution to the linear equation is $\dot{u}=u_0$. Therefore
\begin{equation}
\tilde u_0= u_0\circ\Phi^{-1}.
\end{equation}
It then follows that
\begin{equation}
\nabla u_0=D\Phi \left(\nabla\tilde u_0\circ\Phi\right),
\end{equation}
which we can use in \eqref{linear-uniqueness}, together with  \eqref{eq-p-2}, to get
\begin{multline}
\tilde A=\tilde\gamma|\nabla \tilde u_0|^{p-2}\frac{1}{|D\Phi|\circ\Phi^{-1}}\Bigg(\left[(D\Phi)^TD\Phi\right]\circ\Phi^{-1}
 +(p-2)\frac{|\nabla\tilde u_0|^2}{|\nabla u_0|^2\circ\Phi^{-1}}\\[7pt]\times\left[(D\Phi)^TD\Phi\right]\circ\Phi^{-1}\frac{\nabla\tilde u_0\otimes\nabla\tilde u_0}{|\nabla\tilde u_0|^2} \left[(D\Phi)^TD\Phi\right]\circ\Phi^{-1}\Bigg).
\end{multline}

Let
\begin{equation}\label{eq-F}
\!\!\!F=\left[\frac{(D\Phi)^TD\Phi}{|D\Phi|}\right]\!\circ\Phi^{-1},\;  P=\frac{\nabla\tilde u_0\otimes\nabla\tilde u_0}{|\nabla\tilde u_0|^2}, \; \alpha=|D\Phi|\!\circ\Phi^{-1}\frac{|\nabla\tilde u_0|^2}{|\nabla u_0|^2\!\circ\Phi^{-1}}.
\end{equation}
Note that $F$, $P$ are symmetric matrices, and that $P^2=P$. We have
\begin{equation}
F+\alpha (p-2)FPF=I+(p-2)P.
\end{equation}
As both $I+(p-2)P$ and $F$ are invertible, it follows that both $I+\alpha(p-2)PF$ and $I+\alpha(p-2)FP$ are also invertible.

If we multliply \eqref{eq-F} by $P$ on the left we get
\begin{equation}
PF+\alpha(p-2)PFPF=(p-1)P,
\end{equation}
so
\begin{equation}
PF=(p-1)(I+\alpha(p-2)PF)^{-1}P,
\end{equation}
since the inverse exists, and therefore 
\begin{equation}
PF=PFP.
\end{equation}
On the other hand 
\begin{equation}
FP+\alpha(p-2)FPFP=(p-1)P,
\end{equation}
so
\begin{equation}
FP=(p-1)P(I+\alpha(p-2)FP)^{-1},
\end{equation}
and therefore
\begin{equation}\label{F-P-commute}
FP=PFP=PF.
\end{equation}
Since $F$ and $P$ commute, they can be simultaneously diagonalized, and we can write
\begin{equation}
F=\theta P+\eta (I-P),\quad \theta, \eta \text{ scalars}. 
\end{equation}

It is easy to see that $\eta=1$, while $\theta$ must satisfy the equation
\begin{equation}
\theta+\alpha(p-2)\theta^2=p-1.
\end{equation}
Note that
\begin{equation}
1=\frac{|D\Phi|^2}{|D\Phi|^2}\circ\Phi^{-1}=\det F=\theta\eta=\theta.
\end{equation}
It follows that $F=I$

Suppose a non-trivial diffeomorphism such as $\Phi$ exists. Let $\sigma$ be a scalar conductivity on $\dom$ and let
\begin{equation}
\sigma_*(y)=\frac{\sigma}{|D\Phi|}(D\Phi)^TD\Phi\circ\Phi^{-1}(y)=\sigma\circ\Phi^{-1}(y) F(y)=\sigma\circ\Phi^{-1}(y).
\end{equation}
This new conductivity is also scalar and gives the same DN map as $\sigma$. This violates the known uniqueness results for the Calder\'{o}n problem in the plane. So $\Phi$ must be trivial. Therefore $u_0=\tilde u_0$ and $\gamma=\tilde\gamma$.

\section{Proof of Theorem \ref{thm-main-2}}\label{sec_higher}

As in the previous section, we will denote by $u_0$, $\tilde u_0$, $A$, $\tilde A$, etc. the functions corresponding to the coefficients $\gamma$ and $\tilde \gamma$ respectively. By Proposition \ref{prop-linearization} we have that $\Lambda_A=\Lambda_{\tilde A}$. 

It is an immediate consequence of  \cite[Proposition 1.3]{LeeUh} (or \cite[Theorem 1.3]{KaYun}) that there must exist a neighborhood of $U$ of $\pr\dom$ and a smooth diffeomorphism $\Phi:U\cap\bar\dom\to U\cap\bar\dom$, with $\Phi|_{\pr\dom}=Id$, such that
\begin{equation}\label{linear-uniqueness-boundary}
\left.\pr_\nu^j\tilde A\right|_{\pr\dom}=\left.\pr_\nu^j\frac{1}{|D\Phi|}(D\Phi)^TAD\Phi\right|_{\pr\dom}, \quad j=0,1,2,\ldots.
\end{equation}

Let $z\in\pr\dom$. Unless otherwise specified, all  the following computations will be pointwise, at this point $z$. We wish to proceed inductively in the order of differentiation in \eqref{linear-uniqueness-boundary}.

\paragraph{$0^{th}$ order:} We have that
\begin{equation}\label{linear-uniqueness-boundary-0}
\tilde A(z)=\frac{1}{|D\Phi|(z)}(D\Phi)^T(z)A(z)D\Phi(z).
\end{equation}
If $\tau$ is any unit tangent vector to $\pr\dom$ at $z$, we must have that $D\Phi(z)\tau=\tau$. Since $u_0|_{\pr\dom}=\tilde u_0|_{\pr\dom}$, we also have that $\tau\cdot\nabla u_0(z)=\tau\cdot\nabla\tilde u_0(z)$. Therefore
\begin{equation}
\tau\cdot\tilde A(z)\tau=\tilde\gamma(z)|\nabla\tilde u_0|^{p-2}(z)\left(1+(p-2)\frac{(\tau\cdot\nabla u_0)^2(z)}{|\nabla \tilde u_0|^2(z)}\right).
\end{equation}
On the other hand, by \eqref{linear-uniqueness-boundary-0} we have
\begin{equation}
\tau\cdot\tilde A(z)\tau=\frac{1}{|D\Phi|(z)}\gamma(z)|\nabla u_0|^{p-2}(z)\left(1+(p-2)\frac{(\tau\cdot\nabla u_0)^2(z)}{|\nabla  u_0|^2(z)}\right).
\end{equation}
We can vary $\tau$ in the tangent space to the boundary at $z$, which is at least two dimensional. By choosing $\tau\perp\nabla u_0(z)$ we can separately identify
\begin{equation}
\tilde\gamma(z)|\nabla\tilde u_0|^{p-2}(z)=\frac{1}{|D\Phi|(z)}\gamma(z)|\nabla u_0|^{p-2}(z).
\end{equation}
By possibly shifting the point $z$, or by slightly rotating the vector $\zeta$ in the statement of Proposition \ref{prop-u0-3+}, we can make sure that $\nabla u_0(z)$ is not normal to the boundary. Choosing $\tau\not\perp\nabla u_0(z)$ we get
\begin{equation}
\frac{1}{|\nabla\tilde u_0|^2}(z)=\frac{1}{|\nabla u_0|^2}(z).
\end{equation}
It follows that 
\begin{equation}
|\nabla u_0|(z)=|\nabla\tilde u_0|(z), 
\end{equation}
and, as we already know that $\tau\cdot\nabla u_0(z)=\tau\cdot\nabla\tilde u_0(z)$ for all $\tau$ as above, we also conclude that 
\begin{equation}
\pr_\nu  u_0(z)=\pr_\nu \tilde u_0(z).
\end{equation}
As $\Lambda_\gamma(u_0|_{\pr\dom})=\Lambda_{\tilde\gamma}(u_0|_{\pr\dom})$, we get 
\begin{equation}
\gamma(z)=\tilde\gamma(z).
\end{equation}
This further implies that
\begin{equation}
|D\Phi|(z)=1,
\end{equation}
Since now $A(z)=\tilde A(z)$ and $D\Phi$ acts as the identity in the tangent space to $\pr\dom$ at $z$, equation \eqref{linear-uniqueness-boundary-0} can  only hold if
\begin{equation}
D\Phi(z)=I.
\end{equation}

\paragraph{$1^{st}$ order:} We have that
\begin{equation}\label{linear-uniqueness-boundary-1}
(\pr_\nu\tilde A)(z)=\left(\pr_\nu\frac{1}{|D\Phi|}(D\Phi)^TAD\Phi\right)(z).
\end{equation}

From this point onward, we will use the assumption that $\pr\dom$ is flat in a neighborhood of $z$. For ease of computation, we will rotate our coordinates so that $\nu=e_1$ and locally $\pr\dom\cap U\subset\{x_1=0\}$. We also find it notationally convenient to introduce the tangential gradient $\nabla'=\nabla - \pr_1 e_1$. By possibly shifting the point $z$, or by slightly changing the vector $\zeta$ in the statement of Proposition \ref{prop-u0-3+}, we can make sure that $\pr_1u_0(z)\neq0$.

In the previous step we have shown that $D\Phi(z)=I$. It follows that
\begin{equation}
\pr_j\pr_k\Phi^l(z)=0,\quad\text{unless }j=k=1.
\end{equation}
Rewriting \eqref{linear-uniqueness-boundary-1} with this information, we obtain that at $z$
\begin{equation}\label{linear-uniqueness-boundary-1b}
\pr_1\tilde A_{jk}=\pr_1 A_{jk}+(A_{j1}\pr_1^2\Phi^k+A_{1k}\pr_1^2\Phi^j)-A_{jk}\pr_1^2\Phi^1.
\end{equation}

In preparation for using the above equations, and denoting by $a_{11}$, $a_{jj}$, $a_{j1}$  terms made up of quantities for which uniqueness has already been shown at the previous step, we compute
\begin{multline}
\pr_1 A_{11}=\pr_1\gamma |\nabla u_0|^{p-2}\left(1+(p-2)\frac{(\pr_1 u_0)^2}{|\nabla u_0|^2}\right)\\[7pt]
+\pr_1^2 u_0\gamma\pr_1 u_0|\nabla u_0|^{p-4}(p-2)\left(3+(p-4)\frac{(\pr_1 u_0)^2}{|\nabla u_0|^2}\right)+a_{11}.
\end{multline}
 For $j\neq1$
\begin{multline}
\pr_1 A_{jj}=\pr_1\gamma |\nabla u_0|^{p-2}\left(1+(p-2)\frac{(\pr_j u_0)^2}{|\nabla u_0|^2}\right)\\[7pt]
+\pr_1^2 u_0\gamma\pr_1 u_0|\nabla u_0|^{p-4}(p-2)\left(1+(p-4)\frac{(\pr_j u_0)^2}{|\nabla u_0|^2}\right)+a_{jj}.
\end{multline}
Also
\begin{multline}
\pr_1 A_{j1}=\pr_1\gamma |\nabla u_0|^{p-2}(p-2)\frac{\pr_j u_0\pr_1 u_0}{|\nabla u_0|^2}\\[7pt]
+\pr_1^2 u_0\gamma\pr_j u_0|\nabla u_0|^{p-4}(p-2)\left(1+(p-4)\frac{(\pr_1 u_0)^2}{|\nabla u_0|^2}\right)+a_{j1}.
\end{multline}

Since $\nabla\cdot(\gamma|\nabla u_0|^{p-2}\nabla u_0)=0$, at $z$ we  have that 
\begin{equation}
\pr_1\left[\gamma|\nabla u_0|^{p-2}\pr_1 u_0\right]
=-\nabla'\cdot\left[\gamma|\nabla u_0|^{p-2}\nabla' u_0\right]
=\pr_1\left[\tilde \gamma|\nabla \tilde u_0|^{p-2}\pr_1 \tilde u_0\right],
\end{equation}
by the previous step. Let $\xi_1=\pr_1(\gamma-\tilde\gamma)(z)$, $\xi_2=\pr_1^2( u_0-\tilde u_0)(z)$. It follows that
\begin{equation}
\Theta_{11}\xi_1+\Theta_{12}\xi_2=0,
\end{equation}
where
\begin{equation}
\Theta_{11}=\pr_1 u_0|\nabla u_0|^{p-2},
\end{equation}
\begin{equation}
\Theta_{12}=\gamma |\nabla u_0|^{p-2}\left(1+(p-2)\frac{(\pr_1 u_0)^2}{|\nabla u_0|^2}\right).
\end{equation}

Let $\xi_3=\pr_1^2\Phi^1(z)$. Taking $j=k=1$ in \eqref{linear-uniqueness-boundary-1b}, we obtain the equation
\begin{equation}
\Theta_{21}\xi_1+\Theta_{22}\xi_2+\Theta_{23}\xi_3=0,
\end{equation}
where
\begin{equation}
\Theta_{21}= |\nabla u_0|^{p-2}\left(1+(p-2)\frac{(\pr_1 u_0)^2}{|\nabla u_0|^2}\right)
\end{equation}
\begin{equation}
\Theta_{22}=\gamma\pr_1 u_0|\nabla u_0|^{p-4}(p-2)\left(3+(p-4)\frac{(\pr_1 u_0)^2}{|\nabla u_0|^2}\right)
\end{equation}
\begin{equation}
\Theta_{23}=\gamma |\nabla u_0|^{p-2}\left(1+(p-2)\frac{(\pr_1 u_0)^2}{|\nabla u_0|^2}\right).
\end{equation}


Taking $k=j$ in \eqref{linear-uniqueness-boundary-1b}, we obtain the equation
\begin{equation}
\Theta_{31}\xi_1+\Theta_{32}\xi_2+\Theta_{33}\xi_3=-2\pr_1^2\Phi^j\gamma |\nabla u_0|^{p-2}(p-2)\frac{\pr_j u_0\pr_1 u_0}{|\nabla u_0|^2},
\end{equation}
where
\begin{equation}
\Theta_{31}= |\nabla u_0|^{p-2}\left(1+(p-2)\frac{(\pr_j u_0)^2}{|\nabla u_0|^2}\right),
\end{equation}
\begin{equation}
\Theta_{32}=\gamma\pr_1 u_0|\nabla u_0|^{p-4}(p-2)\left(1+(p-4)\frac{(\pr_j u_0)^2}{|\nabla u_0|^2}\right),
\end{equation}
\begin{equation}
\Theta_{33}=-\gamma |\nabla u_0|^{p-2}\left(1+(p-2)\frac{(\pr_j u_0)^2}{|\nabla u_0|^2}\right),
\end{equation}

Under our assumptions, we are still free to rotate the coordinate axes, as long as the normal direction remains that of $x_1$. We can therefore arrange that $\pr_j u_0(z)=0$. In this case 
we then have the system
\begin{equation}
\left\{\begin{array}{l}\Theta_{11}\xi_1+\Theta_{12}\xi_2=0,\\[7pt] \Theta_{21}\xi_1+\Theta_{22}\xi_2+\Theta_{23}\xi_3=0,\\[7pt] \Theta_{31}\xi_1+\Theta_{32}\xi_2+\Theta_{33}\xi_3=0.\end{array}\right.
\end{equation}
Denoting $\lambda(z)=\gamma^2(z)|\nabla u_0|^{3p-8}(z)$ we compute the determinant
\begin{multline}
{\begin{vmatrix}\Theta_{11} & \Theta_{12}& 0\\[7pt]
\Theta_{21} & \Theta_{22}& \Theta_{23}\\[7pt]
\Theta_{31} & \Theta_{32}& \Theta_{33}\end{vmatrix}}\\[7pt]
=\lambda(z){\scriptsize \begin{vmatrix}
\pr_1 u_0 &|\nabla u_0|^2\left(1+(p-2)\frac{(\pr_1 u_0)^2}{|\nabla u_0|^2}\right) & 0\\[7pt]
1+(p-2)\frac{(\pr_1 u_0)^2}{|\nabla u_0|^2} & (p-2)\pr_1 u_0\left(3+(p-4)\frac{(\pr_1 u_0)^2}{|\nabla u_0|^2}\right) & 1+(p-2)\frac{(\pr_1 u_0)^2}{|\nabla u_0|^2}\\[7pt]
1 & (p-2)\pr_1 u_0 & -1
\end{vmatrix}}\\[7pt]
=\lambda(z){\scriptsize\begin{vmatrix}
\pr_1 u_0 &|\nabla u_0|^2\left(1+(p-2)\frac{(\pr_1 u_0)^2}{|\nabla u_0|^2}\right) & 0\\[7pt]
0 & (p-2)\pr_1 u_0\left(3+(p-4)\frac{(\pr_1 u_0)^2}{|\nabla u_0|^2}\right) & 1+(p-2)\frac{(\pr_1 u_0)^2}{|\nabla u_0|^2}\\[7pt]
2 & (p-2)\pr_1 u_0 & -1
\end{vmatrix}}\\[7pt]
=\lambda(z)\left[2|\nabla u_0|^2\left(1+(p-2)\frac{(\pr_1 u_0)^2}{|\nabla u_0|^2}\right)^2\right.\qquad\qquad\\[7pt]\left.
\qquad\qquad-(p-2)(\pr_1 u_0)^2\left(3+(p-4)\frac{(\pr_1 u_0)^2}{|\nabla u_0|^2}\right)\right]\\[7pt]
=\lambda(z)\left[ 
2|\nabla u_0|^2+(p-2)(\pr_1 u_0)^2+p(p-2)\frac{(\pr_1 u_0)^4}{|\nabla u_0|^2}
\right]\neq 0,
\end{multline}
where the conclusion holds because, since $p>1$, both $p-2>-1$ and $p(p-2)>-1$.
It follows that
\begin{equation}
\pr_1\gamma(z)=\pr_1\tilde\gamma(z),\quad \pr_1^2  u_0(z)=\pr_1^2\tilde u_0(z),\quad  \pr_1^2\Phi^1(z)=0.
\end{equation}

Returning to \eqref{linear-uniqueness-boundary-1b}, with $k=1$, we are left with
\begin{equation}
A_{11}\pr_1^2\Phi^j=0,
\end{equation}
which implies that
\begin{equation}
\pr_1^2\Phi^j(z)=0,
\end{equation}
for all directions $j$ that are orthogonal to the projection of $\nabla u_0$ into the tangent plane. If we choose our coordinates so that the direction of $x_l$ is the same as that of the just mentioned projection, then $A_{l1}(z)\neq 0$, so \eqref{linear-uniqueness-boundary-1b}, with $j=k=l$ gives
\begin{equation}
\pr_1^2\Phi^l(z)=0.
\end{equation}
Therefore we have that
\begin{equation}
\pr_j\pr_k\Phi(z)=0,\quad j,k=1,\ldots,n.
\end{equation}

\paragraph{$m^{th}$ order:} For multi-indices $\alpha\in\N^n$, suppose that we know that
\begin{equation}
\pr^\alpha\gamma(z)=\pr^\alpha\tilde\gamma(z),\quad \alpha_1=0,1,\ldots,m-1,
\end{equation}
\begin{equation}
\pr^\alpha u_0(z)=\pr^\alpha\tilde u_0(z),\quad \alpha_1=0,1,\ldots,m,
\end{equation}
\begin{equation}
\pr^\alpha D\Phi(z)=\pr^\alpha I, \quad \alpha_1=0, 1, \ldots, m-1.
\end{equation}

We have that
\begin{equation}\label{linear-uniqueness-boundary-m}
\left(\pr_1^m\tilde A\right)(z)=\left(\pr_1^m\frac{1}{|D\Phi|}(D\Phi)^TAD\Phi\right)(z).
\end{equation}
Using our induction assumptions, we can rewrite this as
\begin{equation}\label{linear-uniqueness-boundary-mb}
\pr_1^m\tilde A_{jk}=\pr_1^m A_{jk}+(A_{j1}\pr_1^{m+1}\Phi^k+A_{1k}\pr_1^{m+1}\Phi^j)-A_{jk}\pr_1^{m+1}\Phi^1.
\end{equation}

Denoting by $a_{jk}$ terms made up of quantities whose uniqueness follows from the induction hypotheses, we have that
\begin{multline}
\pr_1^m A_{jk}=
\pr_1^m\gamma(z)|\nabla u_0|^{p-2}\left(\delta_{jk}+(p-2)\frac{\pr_j u_0\pr_k u_0}{|\nabla u_0|^2}\right)\\[7pt]
+\pr_1^{m+1} u_0 (p-2)\gamma(z)|\nabla u_0|^{p-4}\\[7pt]
\times\left(\delta_{jk}\pr_1 u_0+(p-4)\frac{\pr_j u_0\pr_k u_0}{|\nabla u_0|^2}\pr_1 u_0+\delta_{1j}\pr_k u_0+\delta_{1k}\pr_j u_0\right)+a_{jk}.
\end{multline}

Since $\nabla\cdot(\gamma|\nabla u_0|^{p-2}\nabla u_0)=0$, at $z$ we  have that 
\begin{multline}
\pr_1^m\left[\gamma|\nabla u_0|^{p-2}\pr_1 u_0\right]\\[7pt]
=-\pr_1^{m-1}\nabla'\cdot\left[\gamma|\nabla u_0|^{p-2}\nabla' u_0\right]
=\pr_1^m\left[\tilde \gamma|\nabla \tilde u_0|^{p-2}\pr_1 \tilde u_0\right].
\end{multline}
This can be rewritten as
\begin{multline}
\pr_1^m(\gamma-\tilde\gamma)\pr_1u_0|\nabla u_0|^{p-2}\\[7pt]+
\pr_1^{m+1}(u_0-\tilde u_0)\gamma|\nabla u_0|^{p-2}\left(1+(p-2)\frac{(\pr_1 u_0)^2}{|\nabla u_0|^2}\right)=0.
\end{multline}
If we set $\xi_1=\pr_1^m(\gamma-\tilde\gamma)(z)$, $\xi_2=\pr_1^{m+1}(u_0-\tilde u_0)$, $\xi_3 = \pr_1^{m+1}\Phi^1$, and we choose a direction $j$ that is orthogonal to the projection of $\nabla u_0(z)$ into the tangent space to $\pr\dom$ at $z$, we obtain the same  system
\begin{equation}
\left\{\begin{array}{l}\Theta_{11}\xi_1+\Theta_{12}\xi_2=0,\\[7pt] \Theta_{21}\xi_1+\Theta_{22}\xi_2+\Theta_{23}\xi_3=0,\\[7pt] \Theta_{31}\xi_1+\Theta_{32}\xi_2+\Theta_{33}\xi_3=0.\end{array}\right.
\end{equation}
Since with our assumptions 
\begin{equation}
\begin{vmatrix}\Theta_{11} & \Theta_{12}& 0\\[7pt]
\Theta_{21} & \Theta_{22}& \Theta_{23}\\[7pt]
\Theta_{31} & \Theta_{32}& \Theta_{33}\end{vmatrix}\neq 0,
\end{equation}
it follows that
\begin{equation}
\pr_1^m\gamma(z)=\pr_1^m\gamma(z),\quad \pr_1^{m+1} u_0(z)=\pr_1^{m+1} u_0(z),\quad \pr_1^{m+1}\Phi^1(z)=0.
\end{equation}

Setting $k=1$ in \eqref{linear-uniqueness-boundary-mb}, we have
\begin{equation}
A_{11}\pr_1^{m+1}\Phi^j=0,
\end{equation}
which implies that
\begin{equation}
\pr_1^{m+1}\Phi^j(z)=0,
\end{equation}
for all directions $j$ that are orthogonal to the projection of $\nabla u_0$ into the tangent plane. If we choose our coordinates so that the direction of $x_l$ is the same as that of the just mentioned projection, then $A_{l1}(z)\neq 0$, so \eqref{linear-uniqueness-boundary-mb}, with $j=k=l$ gives
\begin{equation}
\pr_1^{m+1}\Phi^l(z)=0.
\end{equation}
Therefore 
\begin{equation}
\pr^\alpha D\Phi(z)=\pr^\alpha I,\quad \alpha_1=0,1,\ldots,m.
\end{equation}
This completes the induction step.

\paragraph{Acknowledgments:}
C.C. was supported by NSTC grant number 112-2115-M-A49-002.

\bibliography{pLaplace}
\bibliographystyle{plain}

\end{document}